\documentclass[12]{amsart}
\usepackage{amsmath,amssymb}
\usepackage{amsthm,color,enumerate,comment,centernot,enumitem,url,cite}
\usepackage{graphicx,relsize,bm}
\usepackage{mathtools}
\usepackage{array}
\usepackage{enumitem}
\setenumerate[0]{label=(\alph*)}

\makeatletter
\newcommand{\tpmod}[1]{{\@displayfalse\pmod{#1}}}
\makeatother

\newcommand{\ord}{\operatorname{ord}}

\newtheorem{thm}{Theorem}[section]
\newtheorem{lemma}[thm]{Lemma}

\newtheorem{prop}[thm]{Proposition}

\theoremstyle{remark}

\theoremstyle{definition}

\def\FF {{\mathcal F}}

\def\Z {{\mathbb Z}}

\def\NN {{\mathcal N}}

\def\Q {{\mathbb Q}}

\def\C {{\mathcal C}}

\def\D {{\mathcal D}}

\def\F {{\mathbb F}}
\def\D {{\mathcal D}}
\def\Z {{\mathbb Z}}
\def\Q {{\mathbb Q}}
\def\C {{\mathbb C}}

\def\CC {{\mathcal C}}

\makeatletter
\@namedef{subjclassname@2020}{%
  \textup{2020} Mathematics Subject Classification}
\makeatother

\def\red#1 {\textcolor{red}{#1 }}
\def\blue#1 {\textcolor{blue}{#1 }}

\numberwithin{equation}{section}

\def\Z {{\mathbb Z}}

\begin{document}

\title[A new condition for $k$-Wall-Sun-Sun primes]{A new condition for $k$-Wall-Sun-Sun primes}


\author{Lenny Jones}
\address{Professor Emeritus, Department of Mathematics, Shippensburg University, Shippensburg, Pennsylvania 17257, USA}
\email[Lenny~Jones]{doctorlennyjones@gmail.com}

\date{\today}

\begin{abstract} Let $k\ge 1$ be an integer, and let $(U_n)$
be the Lucas sequence of the first kind defined by
\begin{equation*}\label{Eq:Lucas}
U_0=0,\quad U_1=1\quad \mbox{and} \quad U_n=kU_{n-1}+U_{n-2} \quad \mbox{ for $n\ge 2$}.
\end{equation*}
It is well known that $(U_n)$ is periodic modulo any integer $m\ge 2$, and we let $\pi(m)$
denote the length of this period. A prime $p$ is called a \emph{$k$-Wall-Sun-Sun prime} if $\pi(p^2)=\pi(p)$.

  Let $f(x)\in {\mathbb Z}[x]$ be a monic polynomial of degree $N$ that is irreducible over ${\mathbb Q}$. We say $f(x)$ is \emph{monogenic} if  $\Theta=\{1,\theta,\theta^2,\ldots ,\theta^{N-1}\}$ is a basis for the ring of integers ${\mathbb Z}_K$ of $K={\mathbb Q}(\theta)$, where $f(\theta)=0$. If $\Theta$ is not a basis for ${\mathbb Z}_K$, we say that $f(x)$ is \emph{non-monogenic}.

 Suppose that $k\not \equiv 0 \pmod{4}$ and that ${\mathcal D}:=(k^2+4)/\gcd(2,k)^2$ is squarefree. We prove that $p$ is a $k$-Wall-Sun-Sun prime if and only if ${\mathcal F}_p(x)=x^{2p}-kx^p-1$ is non-monogenic. Furthermore, if $p$ is a prime divisor of $k^2+4$, then ${\mathcal F}_p(x)$ is monogenic.
 \end{abstract}


\maketitle
\section{Introduction}\label{Section:Intro}
Let $k\ge 1$ be an integer, and let $(U_n):=(U_n(k,-1))$ be the Lucas sequence of the first kind defined by
\begin{equation}\label{Eq:Lucas}
U_0=0,\quad U_1=1\quad \mbox{and} \quad U_n=kU_{n-1}+U_{n-2} \quad \mbox{ for $n\ge 2$}.
\end{equation}
It is well known that $(U_n)$ is periodic modulo any integer $m\ge 2$, and we let $\pi(m):=\pi_k(m)$ denote the length of this period. A prime $p$ is called a \emph{$k$-Wall-Sun-Sun prime} if
\begin{equation}\label{Eq:kWSS}
\pi(p^2)=\pi(p).
 \end{equation} Note that $(U_n)$ is the Fibonacci sequence when $k=1$, and in this case, primes satisfying \eqref{Eq:kWSS} are simply called \emph{Wall-Sun-Sun primes}. For the Fibonacci sequence, D. D. Wall \cite{Wall} first asked in 1960 about the existence of primes satisfying \eqref{Eq:kWSS}. In 1992, the Sun brothers \cite{SunSun} showed that the first case of Fermat's Last Theorem for exponent $p$ fails only if $p$ satisfies \eqref{Eq:kWSS}. The question of whether any Wall-Sun-Sun primes exist is still unresolved, and as of December 2022, if $p$ is a Wall-Sun-Sun prime, then $p>2^{64}$ \cite{CDP,Wiki2}. However, the situation is quite different when $k\ge 2$ \cite{Wiki2}.

Several conditions are known to be equivalent to \eqref{Eq:kWSS}. For example, it is easy to see that $U_{\pi(p)}\equiv 0 \pmod{p^2}$ is one such condition. Another, less obvious, equivalent condition is $U_{p-\delta_p}\equiv 0 \pmod{p^2}$, where $\delta_p$ is the Legendre symbol $\left(\frac{k^2+4}{p}\right)$. For more information and proofs, see \cite{Bouazzaoui1,Bouazzaoui2,JonesEJM,Wiki2}.

It is the goal of this article to present a new condition equivalent to \eqref{Eq:kWSS} that is quite unlike any previously known condition. This new condition involves the concept of the monogenicity of a certain polynomial, which we now describe.
Suppose that $f(x)\in \Z[x]$ is a monic polynomial that is irreducible over $\Q$. Let $\Z_K$ be the ring of integers of $K=\Q(\theta)$, where $f(\theta)=0$. Then \cite{Cohen}
\begin{equation} \label{Eq:Dis-Dis}
\Delta(f)=\left[\Z_K:\Z[\theta]\right]^2\Delta(K),
\end{equation}
 where $\Delta(f)$ and $\Delta(K)$ denote, respectively, the discriminants over $\Q$ of $f(x)$ and the number field $K$.
 We define $f(x)$ to be \emph{monogenic} if  $\Theta=\{1,\theta,\theta^2,\ldots ,\theta^{\deg(f)-1}\}$ is a basis for $\Z_K$. If $\Theta$ fails to be a basis for $\Z_K$, we say that $f(x)$ is \emph{non-monogenic}.
Observe then, from \eqref{Eq:Dis-Dis}, that $f(x)$ is monogenic if and only if $\left[\Z_K:\Z[\theta]\right]=1$ or, equivalently, $\Delta(f)=\Delta(K)$.

  The main theorem of this article is as follows: 
  \begin{thm}\label{Thm:Main}
  Let $p$ be a prime. Let $k\ge 1$ be an integer such that $k\not \equiv 0 \pmod{4}$ and $\D$ is squarefree, where
  \begin{equation}\label{Eq:DD}
  {\mathcal D}:=(k^2+4)/\gcd(2,k)^2.
  \end{equation} Then
 $p$ is a $k$-Wall-Sun-Sun prime if and only if
 \[\mbox{the polynomial } \FF_p(x):=x^{2p}-kx^p-1 \mbox{ is non-monogenic.}\]
 Furthermore, if $p$ is a prime divisor of $k^2+4$, then ${\mathcal F}_p(x)$ is monogenic.
  \end{thm}

At first glance, Theorem \ref{Thm:Main} might appear to be just a special case of Theorem 1.2 in \cite{JonesEJM} or Theorem 1.2 in \cite{JonesAJM}. However, upon closer inspection,  we see that  certain restrictions on the prime $p$ and the quadratic character of $\D$ modulo $p$ are necessary in both \cite{JonesEJM} and \cite{JonesAJM}. Therefore, Theorem \ref{Thm:Main} represents an improvement over both \cite{JonesEJM} and \cite{JonesAJM}, in the particular situation of $k$-Wall-Sun-Sun primes, since no such restrictions are required here. Moreover, Theorem \ref{Thm:Main} provides explicit conditions under which ${\mathcal F}_p(x)$ is monogenic.
Since the particular situation of Theorem \ref{Thm:Main} might be more appealing to a broader audience than the generality found in \cite{JonesAJM}, and regardless of the fact that many of the same methods are employed in \cite{JonesAJM}, we give here a self-contained presentation with full details.
\section{Preliminaries}\label{Section:Prelim}
Throughout this article, we assume that $k$ is a positive integer such that $4\nmid k$ and $\D$ is squarefree, where $\D$ is as defined in \eqref{Eq:DD}.
We also let:
\begin{itemize}
  \item $p$ and $q$ denote primes,
  \item $\alpha=\frac{k+\sqrt{k^2+4}}{2}$ and $\beta=\frac{k-\sqrt{k^2+4}}{2}$,
  \item $f(x):=x^2-kx-1$ (the characteristic polynomial of the sequence $(U_n)$),
  \item $\FF_p(x):=x^{2p}-kx^p-1$,
  \item $\ord_m(*)$ denote the order of $*$ modulo the integer $m\ge 2$,
  \item $\delta_p$ denote the Legendre symbol $\left(\frac{k^2+4}{p}\right)$.
 \end{itemize}

   The first result gives some known facts concerning $\pi(p^2)$ and $\pi(p)$. 
\begin{thm}\label{Thm:Period}{\rm \cite{GRS,Renault}}
\text{} 
  \begin{enumerate}
  \item \label{R:I.5} $\pi(p^2)\in \{\pi(p),p\pi(p)\}$.
  \item \label{R:I2} If $\delta_p=1$, then $p-1\equiv 0 \pmod{\pi(p)}$.
  \item \label{R:I3} If $\delta_p=-1$, then $2(p+1)\equiv 0 \pmod{\pi(p)}$.
     \end{enumerate}
\end{thm}


The following lemma is a special case of \cite[Theorem 1.1]{HJBAMS}.
\begin{lemma}\label{Lem:HJ}
  Suppose that $p$ is a divisor of $k^2+4$. If $p=2$, then $p$ is a $k$-Wall-Sun-Sun prime if and only if $k\equiv 0 \pmod{4}$.
  If $p\ge 3$, then $p$ is not a $k$-Wall-Sun-Sun prime.
\end{lemma}
The next two theorems are due to Capelli \cite{S}.
 \begin{thm}\label{Thm:Capelli1}  Let $f(x)$ and $h(x)$ be polynomials in $\Q[x]$ with $f(x)$ irreducible. Suppose that $f(\alpha)=0$. Then $f(h(x))$ is reducible over $\Q$ if and only if $h(x)-\alpha$ is reducible over $\Q(\alpha)$.
 \end{thm}

\begin{thm}\label{Thm:Capelli2}  Let $c\in \Z$ with $c\geq 2$, and let $\alpha\in\C$ be algebraic.  Then $x^c-\alpha$ is reducible over $\Q(\alpha)$ if and only if either there is a prime $p$ dividing $c$ such that $\alpha=\gamma^p$ for some $\gamma\in\Q(\alpha)$ or $4\mid c$ and $\alpha=-4\gamma^4$ for some $\gamma\in\Q(\alpha)$.
\end{thm}

The discriminant of $\FF_p(x)$ given in the next proposition follows from the formula for the discriminant of an arbitrary monic trinomial \cite{Swan}.
\begin{prop}\label{Prop:Swan} 
$\Delta(\FF_p)=(-1)^{(p+1)(2p-1)}p^{2p}(k^2+4)^p$.
\end{prop}

The next theorem is essentially an algorithmic adaptation, specifically for trinomials, of Dedekind's Index Criterion \cite{Cohen}, which is a standard tool used to determine the monogenicity of an irreducible monic polynomial.
\begin{thm}{\rm \cite{JKS2}}\label{Thm:JKS}
Let $N\ge 2$ be an integer.
Let $K=\Q(\theta)$ be an algebraic number field with $\theta\in \Z_K$, the ring of integers of $K$, having minimal polynomial $f(x)=x^{N}+Ax^M+B$ over $\Q$, with $\gcd(M,N)=r$, $N_1=N/r$ and $M_1=M/r$. Let
\begin{equation}\label{Eq:D}
D:=N^{N_1}B^{N_1-M_1}-(-1)^{N_1}M^{M_1}(N-M)^{N_1-M_1}A^{N_1}.
\end{equation} 
 A prime factor $q$ of $\Delta(f)$ does not divide $\left[\Z_K:\Z[\theta]\right]$ if and only if $q$ satisfies one of the following items:
\begin{enumerate}[font=\normalfont]
  \item \label{JKS:I1} when $q\mid A$ and $q\mid B$, then $q^2\nmid B$;
  \item \label{JKS:I2} when $q\mid A$ and $q\nmid B$, then
  \[\mbox{either } \quad q\mid A_2 \mbox{ and } q\nmid B_1 \quad \mbox{ or } \quad q\nmid A_2\left((-B)^{M_1}A_2^{N_1}-\left(-B_1\right)^{N_1}\right),\]
  where $A_2=A/q$ and $B_1=\frac{B+(-B)^{q^e}}{q}$ with $q^e\mid\mid N$;
  \item \label{JKS:I3} when $q\nmid A$ and $q\mid B$, then
  \[\mbox{either } \quad q\mid A_1 \mbox{ and } q\nmid B_2 \quad \mbox{ or } \quad q\nmid A_1B_2^{M-1}\left((-A)^{M_1}A_1^{N_1-M_1}-\left(-B_2\right)^{N_1-M_1
  }\right),\]
  where $A_1=\frac{A+(-A)^{q^j}}{q}$ with $q^j\mid\mid (N-M)$, and $B_2=B/q$;
  \item \label{JKS:I4} when $q\nmid AB$ and $q\mid M$ with $N=uq^m$, $M=vq^m$, $q\nmid \gcd\left(u,v\right)$, then the polynomials
   \begin{align*}
    G(x):&=x^{N/q^m}+Ax^{M/q^m}+B \quad \mbox{and}\\
    H(x):&=\dfrac{Ax^{M}+B+\left(-Ax^{M/q^m}-B\right)^{q^m}}{q}
   \end{align*}
   are coprime modulo $q$;
         \item \label{JKS:I5} when $q\nmid ABM$, then $q^2\nmid D/r^{N_1}$.
   \end{enumerate}
\end{thm}

\section{The Proof of Theorem \ref{Thm:Main}}\label{Section:Main}

We first prove some lemmas. 
\begin{lemma}
  The polynomial $\FF_p(x)$ is irreducible over $\Q$.
\end{lemma}
\begin{proof}
  Clearly, $f(x)$ is irreducible over $\Q$ since $\D$ is squarefree. Note that $f(\alpha)=0$. Let $h(x)=x^{p}$ so that $\FF_p(x)=f(h(x))$. Assume, by way of contradiction, that $f(h(x))$ is reducible. Then, by Theorems \ref{Thm:Capelli1} and \ref{Thm:Capelli2}, we have that $\alpha=\gamma^p$ for some $\gamma\in \Q(\alpha)$.
Then, we see by taking norms that
  \[\NN(\gamma)^p=\NN(\alpha)=-1,\] which implies that $p\ge 3$ and $\NN(\gamma)=-1$, since $\NN(\gamma)\in \Z$. Thus, $\gamma$ is a unit, and therefore $\gamma=\pm \alpha^j$ for some $j\in \Z$, since, in light of the fact that $k\ne 4$, $\alpha$ is the fundamental unit of $\Q(\sqrt{\D})$  \cite{Y}.
   Consequently,
  \[\alpha=\gamma^p=(\pm 1)^p\alpha^{jp},\] which implies that $(\pm 1)^p\alpha^{jp-1}=1$, contradicting the fact that $\alpha$ has infinite order in the group of units of the ring of algebraic integers in the real quadratic field $\Q(\sqrt{\D})$.
\end{proof}

\begin{lemma}\label{Lem:Order}
Suppose that $p\ge 3$. Then
     \begin{enumerate}
     \item \label{Per I:0} $\ord_m(\alpha)=\pi(m)$ for $m\in \{p,p^2\}$, 
     \item \label{Per I:1.5} $\alpha^{p-1}\equiv 1 \pmod{p}$ if $\delta_p=1$,
     \item \label{Per I:2} $\alpha^{p+1}\equiv -1\pmod{p}$ if $\delta_p=-1$.
     \end{enumerate}
\end{lemma}
\begin{proof}
  It follows from \cite{Robinson} that the order, modulo an integer $m\ge 3$, of the companion matrix 
  \[\CC=\left[\begin{array}{cc}
    0&1\\
    1&k
  \end{array}\right]\] for the characteristic polynomial $f(x)$ of $(U_n)$ is $\pi(m)$. Since the eigenvalues of $\CC$ are $\alpha$ and $\beta$, we conclude that
  \[\ord_m\left(\left[\begin{array}{cc}
    \alpha&0\\
    0&\beta
  \end{array}\right]\right)=\ord_m(\CC)=\pi(m), \quad \mbox{for $m\in \{p,p^2\}$.}
  \] It follows that at least one of $\alpha$ and $\beta$ has order $\pi(m)$, and we can assume without loss of generality, that $\ord_{m}(\alpha)=\pi(m)$,
  which establishes \ref{Per I:0}.

For \ref{Per I:1.5} and \ref{Per I:2}, we have by Euler's criterion that
   \[\left(\sqrt{k^2+4}\right)^{p+1}=(k^2+4)^{(p-1)/2}(k^2+4)\equiv \delta_p(k^2+4) \pmod{p},\] which implies that
   $\left(\sqrt{k^2+4}\right)^{p}\equiv \delta_p\sqrt{k^2+4} \pmod{p}$.
   Hence,
   \begin{align*}\label{Eq:Expansion}
     \alpha^{p+1}&=\left(\frac{k+\sqrt{k^2+4}}{2}\right) \left(\frac{k+\sqrt{k^2+4}}{2}\right)^{p}\\
     &=\left(\frac{k+\sqrt{k^2+4}}{2}\right) \sum_{j=0}^p\binom{p}{j}\left(\frac{k}{2}\right)^j\left(\frac{\sqrt{k^2+4}}{2}\right)^{p-j}\\
     &\equiv \left(\frac{k+\sqrt{k^2+4}}{2}\right)\left(\left(\frac{k}{2}\right)^p+\left(\frac{\sqrt{k^2+4}}{2}\right)^{p}\right) \pmod{p}\\
     &\equiv \left(\frac{k+\sqrt{k^2+4}}{2}\right)\left(\frac{k+\delta_p\sqrt{k^2+4}}{2}\right) \pmod{p}\\
     &\equiv \left\{\begin{array}{cl}
      \alpha^2 \pmod{p} & \mbox{if $\delta_p=1$,}\\
     -1 \pmod{p} & \mbox{if $\delta_p=-1$.}
     \end{array} \right.
    \end{align*}
      Since $\alpha\in (\Z/p\Z)^{*}$ when $\delta_p=1$, we note that \ref{Per I:1.5} also follows from Fermat's Little Theorem. 
 \end{proof}

\begin{lemma}\label{Lem:New1}
Suppose that $p\ge 3$. Then
\[\FF_p(\beta)\equiv 0 \pmod{p^2} \ \Longleftrightarrow \ \FF_p(\alpha) \equiv 0 \pmod{p^2}.\]
\end{lemma}
  \begin{proof}
    Note that if $\FF_p(\beta)=\beta^{2p}-k\beta^p-1\equiv 0 \pmod{p^2}$, then
    \begin{equation}\label{Eq:LemNew1}
    \beta^p-k-\beta^{-p}\equiv 0\pmod{p^2}.
    \end{equation} Since $\alpha\beta\equiv -1 \pmod{p}$, we have that $(\alpha\beta)^p\equiv (-1)^p\equiv -1 \pmod{p^2}$. Thus, since $\alpha^p\not \equiv k \pmod{p}$ from Lemma \ref{Lem:Order}, we have
    \begin{align*}
      \FF_p(\beta)\equiv 0 \pmod{p^2}&\Longleftrightarrow \beta^p(\beta^p-k)\equiv 1 \pmod{p^2}\\
      &\Longleftrightarrow \alpha^p\beta^p (\beta^p-k)\equiv \alpha^p \pmod{p^2}\\
      &\Longleftrightarrow -(\beta^p-k)\equiv \alpha^p \pmod{p^2}\\
      &\Longleftrightarrow -(\alpha^p-k)(\beta^p-k)\equiv \alpha^p(\alpha^p-k) \pmod{p^2}\\
      &\Longleftrightarrow -(\alpha^p\beta^p-k\alpha^p-k\beta^p+k^2)\equiv \alpha^p(\alpha^p-k) \pmod{p^2}\\
      &\Longleftrightarrow 1+k(\alpha^p+\beta^p)-k^2 \equiv \alpha^p(\alpha^p-1) \pmod{p^2}\\
      &\Longleftrightarrow 1+k(-\beta^{-p}+\beta^{p})-k^2\equiv \alpha^p(\alpha^p-1) \pmod{p^2}\\
      &\Longleftrightarrow 1\equiv \alpha^p(\alpha^p-1) \pmod{p^2} \quad \mbox{(from \eqref{Eq:LemNew1})}\\
      &\Longleftrightarrow \FF_p(\alpha)\equiv 0 \pmod{p^2}. \qedhere
    \end{align*}
     \end{proof}
  \begin{lemma}\label{Lem:Main1}
Suppose that $p\ge 3$.
Let $\Z_K$ denote the ring of integers of $K=\Q(\theta)$, where $\FF_p(\theta)=0$. Then
\[\FF_p(\alpha)\equiv 0 \pmod{p^2} \quad \Longleftrightarrow \quad [\Z_K:\Z[\theta]]\equiv 0 \pmod{p}.\]
\end{lemma}
\begin{proof}
Since $f(\alpha)=\alpha^2-k\alpha-1=0$, we note that $\alpha^2\equiv k\alpha+1 \pmod{p}$, which implies that
\begin{equation}\label{Eq:alpha}
\alpha^{2p}\equiv (k\alpha+1)^p \pmod{p^2}.
\end{equation}

  Suppose first that $\FF_p(\alpha)=\alpha^{2p}-k\alpha^p-1\equiv 0 \pmod{p^2}$. Observe then that
  \begin{equation}\label{Eq:second2obs}
   -k\alpha^p-1\equiv -\alpha^{2p} \pmod{p^2}.
  \end{equation}
   Let
  \[G(x)=f(x)=x^2-kx-1 \quad \mbox{and}\quad H(x)=\dfrac{-kx^p-1+(kx+1)^p}{p}.\] Hence, $G(\alpha)\equiv 0 \pmod{p}$ and
  \begin{align*}
    pH(\alpha)&=-k\alpha^p-1+(k\alpha+1)^p\\
         &\equiv -\alpha^{2p}+(k\alpha+1)^p \pmod{p^2} \quad \mbox{(from \eqref{Eq:second2obs})}\\
         &\equiv -\alpha^{2p}+\alpha^{2p}\pmod{p^2} \quad \mbox{(from \eqref{Eq:alpha})}\\
         &\equiv 0 \pmod{p^2}.
  \end{align*} Thus, $G(x)$ and $H(x)$ are not coprime modulo $p$ so that $[\Z_K:\Z[\theta]]\equiv 0 \pmod{p}$ by item \ref{JKS:I4} of Theorem \ref{Thm:JKS}.

  Conversely, suppose that $[\Z_K:\Z[\theta]]\equiv 0 \pmod{p}$. Then, we have by item \ref{JKS:I4} of Theorem \ref{Thm:JKS} that $G(x)$ and $H(x)$ are not coprime modulo $p$. In light of Lemma \ref{Lem:New1}, we assume then, without loss of generality, that
  \begin{equation}\label{Eq:alphaH}
  pH(\alpha)=-k\alpha^p-1+(k\alpha+1)^p \equiv 0 \pmod{p^2}.
  \end{equation}
  Hence,
  \begin{align*}
    \FF_p(\alpha)&=\alpha^{2p}-k\alpha^p-1\\
    &\equiv (k\alpha+1)^p-k\alpha^p-1 \quad \mbox{(from \eqref{Eq:alpha})}\\
    &\equiv (k\alpha^p+1)-k\alpha^p-1 \pmod{p^2} \quad \mbox{(from \eqref{Eq:alphaH})}\\
    &\equiv  0\pmod{p^2},
  \end{align*}
  which completes the proof.
\end{proof}

\begin{lemma}\label{Lem:Main2}
  Suppose that $p\ge 3$. Then 
  \[\mbox{$p$ is a $k$-Wall-Sun-Sun prime} \quad \Longleftrightarrow \quad \FF_p(\alpha) \equiv 0\pmod{p^2}.\]
  \end{lemma}
\begin{proof}
We consider the three cases: $\delta_p\in \{0,-1,1\}$.

Suppose first that $\delta_p=0$. Then $k^2+4 \equiv 0 \pmod{p}$, so that $\alpha \equiv k/2 \pmod{p}$ and $(k/2)^2\equiv -1\pmod{p}$. Hence, $(k/2)^{2p}\equiv -1 \pmod{p^2}$ or, equivalently,
 \begin{equation}\label{Eq:equiv}
 k^{2p}\equiv -2^{2p}\pmod{p^2}.
 \end{equation} By Lemma \ref{Lem:HJ}, we have that $p$ is not a $k$-Wall-Sun-Sun prime. We must show that $\FF_p(\alpha)\not \equiv 0\pmod{p^2}$. Assume, by way of contradiction, that
 \[\FF_p(\alpha)\equiv (k/2)^{2p}-k(k/2)^p-1\equiv -1-k(k/2)^p-1\equiv 0 \pmod{p^2}.\] Thus,
 \begin{equation}\label{Eq:kp+1}
 k^{p+1}\equiv -2^{p+1}\pmod{p^2}.
 \end{equation} Squaring both sides of \eqref{Eq:kp+1} yields
 \begin{equation}\label{Eq:equiv1}
 k^2(k^{2p})\equiv -4(-2^{2p})\pmod{p^2}.
 \end{equation} Note that $p\nmid k$ since $p\ge 3$. Therefore, $k^2+4\equiv 0 \pmod{p^2}$ 
  from \eqref{Eq:equiv} and \eqref{Eq:equiv1}, which contradicts the fact that $\D$ is squarefree, and completes the proof when $\delta_p=0$.

Suppose next that $\delta_p=-1$. Assume first that $p$ is a $k$-Wall-Sun-Sun prime.
      Then, since $\pi(p^2)=\pi(p)$, we conclude from item \ref{R:I3} of Theorem \ref{Thm:Period}, and items \ref{Per I:0} and \ref{Per I:2} of Lemma \ref{Lem:Order} that
       \begin{equation}\label{Eq:modp^2}
       (\alpha^{p+1}-1)(\alpha^{p+1}+1)\equiv \alpha^{2(p+1)}-1\equiv 0 \pmod{p^2}.
       \end{equation} Note that $\alpha^{p+1}-1\not \equiv 0\pmod{p}$ since $\alpha^{p+1}+1\equiv 0\pmod{p}$ from item \ref{Per I:2} of Lemma \ref{Lem:Order}. Therefore, we see from \eqref{Eq:modp^2} that $\alpha^{p+1}+1\equiv 0\pmod{p^2}$, or equivalently, that $\alpha^p\equiv -\alpha^{-1} \pmod{p^2}$. Hence,
       \[\FF_p(\alpha)=\alpha^{2p}-k\alpha^p-1\equiv \alpha^{-2}+k\alpha^{-1}-1\equiv -\frac{\alpha^2-k\alpha-1}{\alpha^2}\equiv 0 \pmod{p^2}.\]

       Conversely, assume that $\FF_p(\alpha)\equiv 0 \pmod{p^2}$.
       Since $\delta_p=-1$, we have that $f(x)$ is irreducible modulo $p$. Consequently, the only zeros of $f(x)$ in $(\Z/p^2\Z)[\sqrt{\D}]$ are $\alpha$ and $\beta=-\alpha^{-1}$.
Hence,
\[\mbox{either} \quad \alpha^p\equiv \alpha \pmod{p^2} \quad \mbox{or}\quad \alpha^p\equiv \beta\pmod{p^2}.\]
If $\alpha^p\equiv \alpha \pmod{p^2}$, then, from item \ref{Per I:2} of Lemma \ref{Lem:Order}, we have that
\[\frac{k^2+2+k\sqrt{k^2+4}}{2}= \alpha^2+1\equiv \alpha^{p+1}+1\equiv 0 \pmod{p},\]
which implies that $k^2+2\equiv 0 \pmod{p}$, and either $p\mid k$ or $k^2+4\equiv 0 \pmod{p}$. In either case, we arrive at the contradiction that $p=2$.  Hence,
\[\alpha^p\equiv \beta\equiv -\alpha^{-1}\pmod{p^2}\quad \mbox{or equivalently,} \quad  \alpha^{p+1}\equiv -1\pmod{p^2}.\]
Thus, $\alpha^{2(p+1)}\equiv 1 \pmod{p^2}$ so that
\begin{equation}\label{Eq:orderdiv}
2(p+1)\equiv 0 \pmod{\ord_{p^2}(\alpha)}.
\end{equation} By item \ref{Per I:0} of Lemma \ref{Lem:Order} and item \ref{R:I.5} of Theorem \ref{Thm:Period}, we have that
\[\ord_{p^2}(\alpha)=\pi(p^2)\in \{\pi(p),p\pi(p)\}.\]
Therefore, we see that $\pi(p^2)=p\pi(p)$ is impossible since $p^2-1\not \equiv 0 \pmod{p}$. Consequently, $\pi(p^2)=\pi(p)$, which implies that $p$ is a $k$-Wall-Sun-Sun prime.

Finally, suppose that $\delta_p=1$. Assume first that $p$ is a $k$-Wall-Sun-Sun prime. Since $\pi(p^2)=\pi(p)$, it follows from item \ref{R:I2} of Theorem \ref{Thm:Period}, and items \ref{Per I:0} and  \ref{Per I:1.5} of Lemma \ref{Lem:Order} that
\[\alpha^{p-1}\equiv 1 \pmod{p^2}\quad \mbox{or equivalently,} \quad \alpha^p\equiv \alpha \pmod{p^2}.\]
Thus, since $f(\alpha)=\alpha^2-k\alpha-1=0$, we have that
\[\FF_p(\alpha)=\alpha^{2p}-k\alpha^p-1\equiv \alpha^2-k\alpha-1\equiv 0\pmod{p^2}.\]

Conversely, assume that $\FF_p(\alpha)\equiv 0 \pmod{p^2}$. Since $f(\alpha)=\alpha^2-k\alpha-1=0$, we have that
\begin{equation}\label{Eq:frac}
\alpha+1/\alpha=2\alpha-k.
\end{equation} Additionally, note that
\[\widehat{\alpha}=\alpha-\frac{f(\alpha)}{f'(\alpha)}=\alpha-\frac{\alpha^2-k\alpha-1}{2\alpha-k}=\frac{\alpha^2+1}{2\alpha-k}\]
 is the Hensel lift modulo $p^2$ of $\alpha$, so that $f(\widehat{\alpha})\equiv 0 \pmod{p^2}$. Then, since
 \[\FF_p(\alpha)=(\alpha^p)^2-k(\alpha^p)-1\equiv 0 \pmod{p^2},\] it follows that
\[\alpha^p\equiv \frac{\alpha^2+1}{2\alpha-k}\pmod{p^2},\] which implies that
\[\alpha^{p-1}\equiv \frac{\alpha+1/\alpha}{2\alpha-k}\equiv 1 \pmod{p^2},\]
from \eqref{Eq:frac}. Hence, $p-1\equiv 0 \pmod{\ord_{p^2}(\alpha)}$. By item \ref{Per I:0} of Lemma \ref{Lem:Order} and item \ref{R:I.5} of Theorem \ref{Thm:Period}, we have that
\[\ord_{p^2}(\alpha)=\pi(p^2)\in \{\pi(p),p\pi(p)\}.\]
Therefore, we see that $\pi(p^2)=p\pi(p)$ is impossible since $p-1\not \equiv 0 \pmod{p}$. Consequently, $\pi(p^2)=\pi(p)$, which implies that $p$ is a $k$-Wall-Sun-Sun prime.
\end{proof}

Combining Lemma \ref{Lem:Main1} and Lemma \ref{Lem:Main2} yields the following.
\begin{lemma}\label{Lem:Main3}
Suppose that $p\ge 3$.
Let $\Z_K$ denote the ring of integers of $K=\Q(\theta)$, where $\FF_p(\theta)=0$. Then
\[\mbox{$p$ is a $k$-Wall-Sun-Sun prime} \quad \Longleftrightarrow \quad [\Z_K:\Z[\theta]]\equiv 0 \pmod{p}.\]
\end{lemma}


We are now in a position to provide a proof of the main result.
 \begin{proof}[Proof of Theorem \ref{Thm:Main}] We first investigate the monogenicity of $\FF_p(x)$. Let $\Z_K$ denote the ring of integers of $K=\Q(\theta)$, where $\FF_p(\theta)=0$.  Recall from Proposition \ref{Prop:Swan} that
 \[\Delta(\FF_p)=(-1)^{(p+1)(2p-1)}p^{2p}(k^2+4)^p.\]
 Let $q\ne p$ be a prime divisor of $\Delta(\FF_p)$. Then $k^2+4\equiv 0 \pmod{q}$.
  Suppose first that $q\ge 3$. Then $q\nmid kp$, and we use item  \ref{JKS:I5} of Theorem \ref{Thm:JKS} to address $q$. Since $\D$ is squarefree, we deduce that $q^2\nmid D/p^2$, and therefore, $[\Z_K:\Z[\theta]]\not \equiv 0 \pmod{q}$. Suppose next that $q=2$. Then $2\mid k$, and we use item \ref{JKS:I2} of Theorem \ref{Thm:JKS} to address $q$. Since $B_1=0$, the first condition fails. However, since $4\nmid k$, we see that $2\nmid A_2$, and so the second condition is satisfied. Hence,  $[\Z_K:\Z[\theta]]\not \equiv 0 \pmod{2}$.

  Thus, we have shown that the monogenicity of $\FF_p(x)$ is completely determined by the prime $p$. More explicitly, we have that
  \[\FF_p(x) \mbox{ is monogenic}\ \Longleftrightarrow \ [\Z_K:\Z[\theta]]\not \equiv 0 \pmod{p}.\] Consequently, if $p\ge 3$, then the theorem follows from Lemma \ref{Lem:Main3}.

  We now address the case $p=2$. Recall that $4\nmid k$. We examine the two subcases: $k\equiv 2 \pmod{4}$ and $k\equiv 1 \pmod{2}$.

  If $k\equiv 2\pmod{4}$, then $k^2+4 \equiv 0 \pmod{2}$ and $p=2$ is not a $k$-Wall-Sun-Sun prime by Lemma \ref{Lem:HJ}.  Since $2\mid k$, we apply item \ref{JKS:I2} of Theorem \ref{Thm:JKS}, and use the same argument as used above, to deduce that $[\Z_K:\Z[\theta]]\not \equiv 0 \pmod{2}$. Therefore, the theorem is established when $p=2$ and $k\equiv 2\pmod{4}$.

   If $k\equiv 1 \pmod{2}$, then straightforward computations reveal that $\pi(4)=6$ and $\pi(2)=3$.
   Hence, $p=2$ is not a $k$-Wall-Sun-Sun prime in this subcase as well, and we must show that $\FF_2(x)$ is monogenic.
 We use item \ref{JKS:I4} of Theorem \ref{Thm:JKS} with $q=p=2$ to see that
  \[G(x)=x^2-kx-1 \quad \mbox{and} \quad H(x)=\frac{-kx^2-1+(kx+1)^2}{2}=kx\left(\frac{k-1}{2}x+1\right).\] Since $G(x)$ is irreducible in $\F_2[x]$, it follows that $G(x)$ and $H(x)$ are coprime in $\F_2[x]$. Hence, $\FF_2(x)$ is monogenic in this case, which completes the proof of the main statement of the theorem. 
  
  Furthermore, it then follows immediately from Lemma \ref{Lem:HJ} that $\FF_p(x)$ is monogenic if $p$ is a prime divisor of $k^2+4$. 
    \end{proof}



\end{document}